\documentclass{amsart}

\usepackage{amssymb}
\usepackage{amsthm}
\usepackage[all]{xy}

\newtheorem{lemma}{Lemma}
\newtheorem{theorem}[lemma]{Theorem}
\newtheorem{corollary}[lemma]{Corollary}
\newtheorem{proposition}[lemma]{Proposition}

\newcommand*{\pd}{\mathop{\mathrm{pd}}\displaylimits}

\title[Unions of chains of ideals]{On the unions of ascending chains of direct sums of ideals of $h$-local Pr\"{u}fer domains}
\author{J. E. Mac\'{\i}as-D\'{\i}az}
\address{Departamento de Matem\'{a}ticas y F\'{\i}sica, Universidad Aut\'{o}noma de Aguascalientes, Avenida Universidad 940, Ciudad Universitaria, Aguascalientes, Ags. 20100, Mexico}
\email{jemacias@correo.uaa.mx}

\date{\today}

\begin{document}

\begin{abstract}
In this work, we investigate conditions under which unions of ascending chains of modules which are isomorphic to direct sums of ideals of an integral domain are again isomorphic to direct sums of ideals. We obtain generalizations of the Pontryagin-Hill theorems for modules which are direct sums of ideals of $h$-local Pr\"{u}fer domains. Particularly, we prove that a torsion-free module over a Dedekind domain with a countable number of maximal ideals is isomorphic to a direct sum of ideals if it is the union of a countable ascending chain of pure submodules which are isomorphic to direct sums of ideals.
\end{abstract}

\maketitle

\section{Introduction\label{Sec:Intro}}

In the last century, Lev Pontryagin and Paul Hill studied conditions under which torsion-free abelian groups are free. In their investigations, the concept of purity of subgroups was crucial. More precisely, a subgroup $H$ of an abelian group $G$ is \emph {pure} if every equation of the form $k x = a \in H$, with $k \in \mathbb {Z}$, is solvable in $H$ whenever it is solvable in $G$. Equivalently, solubility in $G$ of each system of equations of the form
\begin{equation}
\sum _{j = 1} ^m k _{i j} x _j = a _i \in H \quad (i = 1 , \dots , n),
\end{equation}
with every $k _{i j} \in \mathbb {Z}$, implies its solubility in $H$.

In $1934$, Pontryagin proved that a countable, torsion-free abelian group is free if and only if every finite rank, pure subgroup is free \cite {Pontryagin}. Equivalently, every properly ascending chain of pure subgroups of the same finite rank is finite. From the proof of this result, it follows that a torsion-free abelian group $G$ is free if there exists an ascending chain
\begin{equation}
0 = G _0 \leq G _1 \leq \dots \leq G _n \leq \dots \quad (n < \omega) \label{CountChain}
\end{equation}
consisting of pure subgroups of $G$ whose union is equal to $G$, such that every $G_n$ is free and countable.

Later, in $1970$, Hill established that, in order for an abelian group $G$ to be free, it is sufficient that it be the union of a countable ascending chain (\ref {CountChain}) of free, pure subgroups \cite {Hill}. In other words, Hill proved that the condition of countability on the cardinality of the links $G _n$ in Pontryagin's theorem was superfluous. The proof of this theorem relies on some important facts about commutative groups, one of them being that subgroups of torsion-free abelian groups can be embedded in pure subgroups of the same rank. Applications of these criteria may be actually found in a variety of algebraic results \cite{Cornelius, Eklof, Shelah, Mekler}.

In view of the importance of the Pontryagin-Hill theorems in algebra, it is highly desirable to explore the possibility to generalize these criteria to more general scenarios. In this article, we generalize those results to modules which are isomorphic to direct sums of ideals of $h$-local Pr\"{u}fer domains. Section \ref {Sec:Balanc} introduces the concept of balanced submodules and provides some useful criteria for balancedness. Section \ref {Sec:h-local} serves as an introduction to $h$-local Pr\"{u}fer domains and their properties, while Section \ref {Sec:Main} presents the most important theorems of this work. 

\section{Balancedness\label{Sec:Balanc}}

Once and for all we declare that, throughout this work, $R$ will represent an integral domain unless stated otherwise. Modules are understood to be over $R$ when no other statement is done.

A submodule $N$ of an $R$-module $M$ is \emph {relatively divisible} if the inclusion $N \cap r M \leq r N$ holds, for every $r \in R$. Equivalently, solubility in $M$ of equations of the form $r x = a \in N$, with $r \in R$, implies their solubility in $N$. We say that $N$ is \emph {pure} in $M$ if every finite system of equations
\begin{equation}
\sum _{j = 1} ^m r _{i j} x _j = a _i \in N \quad (i = 1 , \dots , n),
\end{equation}
with $r _{i j} \in R$, is solvable in $N$ whenever it is solvable in $M$. Under these circumstances, a short-exact sequence $0 \rightarrow N \rightarrow M \rightarrow Q \rightarrow 0$ is $RD$-exact (respectively, pure-exact) if $N$ is a relatively divisible (respectively, pure) submodule of $M$. Evidently, purity implies relative divisibility, and they both coincide for modules over Pr\"{u}fer domains \cite {Warfield}, that is, integral domains in which finitely generated ideals are projective. Moreover, Pr\"{u}fer domains are the only integral domains for which relative divisibility and purity are equivalent \cite{Cartan-Eilenberg}.

A submodule $A$ of the $R$-module $B$ is \emph {balanced} if $B / A$ is torsion-free, and every rank $1$, torsion-free $R$-module $J$ has the projective property with respect to the short-exact sequence $0 \rightarrow A \rightarrow B \rightarrow B / A \rightarrow 0$. In other words, for every homomorphism $\phi$ from $J$ into $B / A$, there exists a homomorphism $\psi$ from $J$ into $B$, which makes the following diagram commute:
\begin{equation}
\xymatrix{
 & & & J \ar[d]^\phi\ar[ld]_\psi & \\
0 \ar[r] & A \ar[r] & B\ar[r] & B / A\ar[r] & 0}
\end{equation}
In this context, a short-exact sequence $0 \rightarrow A \rightarrow B \rightarrow C \rightarrow 0$ of $R$-modules, with $A$ balanced in $B$ and $C$ torsion free, is called \emph {balanced-exact}. Clearly, direct sums of rank $1$ modules have the projective property with respect to balanced-exact sequences. For a list of relevant properties of relative divisibility, purity and balancedness of modules, we refer to \cite {Fuchs-Salce2}.

\begin{lemma}
Let $L$ be a pure submodule of the torsion-free module $M$, with the property that $L$ is balanced in $N$, for every $L \leq N \leq M$ such that $N / L$ has rank $1$. Then, $L$ is balanced in $M$. \label{Lemma:5.1}
\end{lemma}

\begin{proof}
Let $J$ be a rank $1$, torsion-free module, and let $\phi$ be a homomorphism from $J$ into $M / L$. The image of $J$ under $\phi$ is a submodule $N / L$ of $M / L$ of rank at most $1$. By hypothesis, there exists a homomorphism $\psi$ from $J$ into $L$, such that the following diagram with exact rows commutes:
\begin{equation}
\xymatrix{
 & & & J \ar[d]^{\phi ^\prime}\ar[ld]_\psi & \\
0 \ar[r] & L \ar[r]\ar@{=}[d] & N\ar[r]\ar[d]^\iota & N / L\ar[r]\ar[d]^{\iota ^\prime} & 0 \\
0 \ar[r] & L \ar[r] & M\ar[r] & M / L\ar[r] & 0}
\end{equation}
Here, the homomorphisms $\iota$ and $\iota ^\prime$ are inclusions, and $\phi ^\prime$ is the restriction of $\phi$ onto $N / L$. We conclude that $J$ is projective with respect to the sequence of the bottom row and, consequently, that $L$ is balanced in $M$.
\end{proof}

We close this section with another criterion for balancedness.

\begin{proposition}
Consider the following commutative diagram with exact rows and torsion-free modules:
\begin{equation}
\xymatrix{
0 \ar[r] & L \ar[r]\ar[d] & N\ar[r]^\alpha\ar[d]^\mu & M\ar[r]\ar[d]^\tau & 0 \\
0 \ar[r] & L ^\prime \ar[r] & N ^\prime\ar[r]^\beta & M ^\prime\ar[r] & 0} \label{Eq:Diag5.2}
\end{equation}
If there exists a homomorphism $\rho$ from $M ^\prime$ to $M$, such that $\tau \rho = 1$, and if the top row of (\ref {Eq:Diag5.2}) is balanced-exact, then the bottom row is also balanced-exact. \label{Lemma:5.2}
\end{proposition}

\begin{proof}
Let $\phi$ be a homomorphism from a rank $1$, torsion-free module $J$ into $M ^\prime$. Then, there exists a homomorphism $\sigma$ from $J$ into $N$, such that $\alpha \sigma = \rho \phi$. Clearly, the homomorphism $\psi = \mu \sigma$ has the property that $\beta \psi = \phi$, and we conclude that the bottom row is balanced-exact.
\end{proof}

\section{$h$-local domains\label{Sec:h-local}}

An \emph {$h$-local domain} is an integral domain $R$ with the following properties:
\begin{enumerate}
\item[(i)] every nonzero prime ideal of $R$ is contained in exactly one maximal ideal, and
\item[(ii)] every nonzero element of $R$ is contained in all but a finite number of maximal ideals.
\end{enumerate}
A \emph {valuation domain} is an integral domain where ideals form a chain under inclusion; clearly, valuation domains have a unique maximal ideal. By a \emph {Dedekind domain} we mean a hereditary domain, that is, a domain where all the ideals are projective. These two types of domains are examples of $h$-local Pr\"{u}fer domain.

Given two ideals $I$ and $J$ of an integral domain $R$ with field of quotients $Q$, the \emph {residual} of $I$ modulo $J$ is defined by $I : J = \{ a \in Q : a J \leq I \}$. Indeed, Olberding proved \cite {Olberding} that a Pr\"{u}fer domain $R$ is $h$-local if and only if $(J + K) : I = (J : I) + (K : I)$, for any ideals $I$, $J$ and $K$ of $R$. With this characterization, he proves the following result, which generalizes the well-known fact that, over valuation domains and Dedekind domains, pure submodules of modules which are finite direct sums of ideals are direct summands isomorphic to direct sums of ideals.

\begin{lemma}[Olberding \cite {Olberding}]
Let $R$ be an $h$-local Pr\"{u}fer domain, and let $A$ be a pure submodule of an $R$-module $B$ which is a finite direct sum of ideals of $R$. Then,
\begin{enumerate}
\item[\rm (i)] $A$ is a summand of $B$, and
\item[\rm (ii)] $A$ is isomorphic to a direct sum of ideals of $R$. \qed
\end{enumerate} \label{Lemma:6.2}
\end{lemma}

\begin{lemma}
Let $R$ be an $h$-local Pr\"{u}fer domain. Every pure submodule of a module which is a countable direct sum of ideals of $R$ is isomorphic to a direct sum of ideals of $R$. \label{Lemma:5.4}
\end{lemma}

\begin{proof}
Let us assume that $M$ is the direct sum of the ideals $I _n$ of $R$, with $n < \omega$, and let $A$ be a pure submodule of $M$. Fix a maximal independent set $\{ a _n : n \in \mathbb {Z} ^+ \}$ of $A$, and assume that, for some $n < \omega$, we have already constructed the finite ascending chain $0 = A _0 \leq A _1 \leq \dots \leq A _n$ of submodules of $A$, satisfying the following properties, for every $i < n$:
\begin{enumerate}
\item[(a)] $A _{i + 1}$ contains $\{ a _1 , \dots , a _{i + 1} \}$,
\item[(b)] $A _{i + 1}$ is a finite rank, pure submodule of $A$, and
\item[(c)] $A _{i + 1} = A _i \oplus B _i$, for some submodule $B _i$ of $A _{i + 1}$ isomorphic to a direct sum of ideals of $R$.
\end{enumerate}

Take a maximal independent set $Y _n$ of $A _n$, and fix a finite direct sum $N$ of ideals of $R$ in the decomposition of $M$, which contains the set $Y \cup \{ a _{n + 1} \}$. Clearly, the purification of this set has finite rank, is contained in $N$ and, by Lemma \ref {Lemma:6.2}, is a finite direct sum of ideals of $R$. Moreover, $A _{n + 1} = A _n \oplus B _n$, for some submodule $B _n$ of $A _{n + 1}$ isomorphic to a direct sum of ideals of $R$. By induction, $A$ is the union of the countable ascending chain
\begin{equation}
0 = A _0 \leq A _1 \leq \dots \leq A _n \leq \dots \quad (n < \omega). \label {Eq:Chain6.11}
\end{equation}
We conclude that $A$ is isomorphic to the direct sum of the modules $B _n$, for $n < \omega$. Thus, $A$ is isomorphic to a direct sum of ideals of $R$.
\end{proof}

\begin{lemma}
Let $R$ be an $h$-local Pr\"{u}fer domain, let $M$ be a direct sum of ideals of $R$, and let $A$ be a pure submodule of $M$. If $A$ is the direct sum of countable rank submodules, then it is isomorphic to a direct sum of ideals of $R$. \label {Lemma:6.5}
\end{lemma}

\begin{proof}
Assume that $A$ is the direct sum of countable rank submodules $A _\beta$, where $\beta$ runs in a set of indexes $\Lambda$, and let $A$ be a pure submodule of $M = \oplus _{\alpha < \Omega} I _\alpha$, where every $I _\alpha$ is an ideal of $R$. Then, every $A _\beta$ is contained as a pure submodule in a countable direct sum of ideals of $R$ in the decomposition of $M$. By Lemma \ref {Lemma:5.4}, every $A _\beta$ is isomorphic to a direct sum of ideals of $R$ and, consequently, $A$ is likewise isomorphic to a direct sum of ideals of $R$.
\end{proof}

For our next result, we employ the well-known theorem by Kaplansky which states that every direct summand of a direct sum of countable rank modules is also a direct sum of countable rank modules \cite {Kaplansky}.

\begin{theorem}
Let $R$ be an $h$-local Pr\"{u}fer domain. Every direct summand of a module which is a direct sum of ideals of $R$ is isomorphic to a direct sum of ideals of $R$.
\end{theorem}

\begin{proof}
If $M$ is a direct sum of ideals of $R$, then it is a direct sum of countable rank submodules. If $A$ is a direct summand of $M$, then it is pure in $M$ and a direct sum of countable rank submodules. The conclusion is achieved now by means of Lemma \ref {Lemma:6.5}.
\end{proof}

\section{Main results\label{Sec:Main}}

Recall that a continuous, well-ordered, ascending chain of a module $M$ is an ascending chain
\begin{equation}
0 = A _0 \leq A _1 \leq \dots \leq A _\alpha \leq A _{\alpha + 1} \leq \dots \quad (\alpha < \tau) \label {Eq:Chain6.9}
\end{equation}
of submodules of $M$ such that $A _\alpha = \bigcup _{\gamma < \alpha} A _\gamma$, for every limit ordinal $\alpha < \tau$.

\begin{lemma}
A torsion-free $R$-module $A$ is isomorphic to a direct sum of ideals of $R$ if it is the union of a continuous, well-ordered, ascending chain \eqref {Eq:Chain6.9} of submodules, such that the following properties are satisfied, for every $\alpha < \tau$:
\begin{enumerate}
\item[\rm (i)] $A _\alpha$ is a balanced submodule of $A _{\alpha + 1}$, and
\item[\rm (ii)] $A _{\alpha + 1} / A _\alpha$ is isomorphic to a direct sum of ideals of $R$.
\end{enumerate} \label {Lemma:6.9}
\end{lemma}

\begin{proof}
For every $\alpha < \tau$, the balanced-exact sequence $0 \rightarrow A _\alpha \rightarrow A _{\alpha + 1} \rightarrow A _{\alpha + 1} / A _\alpha \rightarrow 0$ splits. So, there exists a submodule $B _\alpha$ of $A _{\alpha + 1}$ which is isomorphic to a direct sum of ideals of $R$, such that $A _{\alpha + 1} = A _\alpha \oplus B _\alpha$. Then, $A$ is isomorphic to the direct sum of the modules $B _\alpha$ and, so, isomorphic to a direct sum of ideals of $R$.
\end{proof}

The following is a generalization of Pontryagin's criterion of freeness to modules which are isomorphic to direct sums of ideals of an $h$-local Pr\"{u}fer domain.

\begin{theorem}
Let $R$ be an $h$-local Pr\"{u}fer domain. A countable rank, torsion-free module is isomorphic to a direct sum of ideals of $R$ if and only if every finite rank, pure submodule is isomorphic to a direct sum of ideals of $R$.
\end{theorem}

\begin{proof}
Let $M$ be a countable rank, torsion-free module, and assume that it is isomorphic to a direct sum of ideals of $R$. If $A$ is a finite rank, pure submodule of $M$, then it is contained in a finite direct sum $B$ of ideals in the decomposition of $M$. Lemma \ref {Lemma:6.2} implies that $A$ is isomorphic to a direct sum of ideals of $R$.

Conversely, let $\{ a _n : n \in \mathbb {Z} ^+ \}$ be a maximal independent set in $M$. For every positive integer $n$, let $A _n$ be the purification of $\{ a _1 , \dots , a _n \}$ in $M$. Then, each $A _n$ is a finite rank, pure submodule of $M$ and, by hypothesis, isomorphic to a direct sum of ideals of $R$. In such a way, we construct a countable ascending chain (\ref {Eq:Chain6.11}) of pure submodules of $M$ which are isomorphic to finite direct sums of ideals of $R$. Clearly, $M$ is equal to the union of the links of (\ref {Eq:Chain6.11}). Moreover, Lemma \ref {Lemma:6.2} yields that, for every $n < \omega$, there exists a submodule $B _n$ of $A _{n + 1}$ which is isomorphic to a finite direct sum of ideals of $R$, such that $A _{n + 1} = A _n \oplus B _n$. Consequently, $M$ is isomorphic to a direct sum of ideals of $R$.
\end{proof}

\begin{theorem}
Let $R$ be an $h$-local Pr\"{u}fer domain. A torsion-free module $M$ is isomorphic to a direct sum of ideals of $R$ if it is the union of a countable ascending chain
\begin{equation}
0 = M _0 \leq M _1 \leq \dots \leq M _n \leq \dots \quad (n < \omega) \label {Eq:Chain6.5}
\end{equation}
of submodules, such that the following properties are satisfied, for every $n < \omega$:
\begin{enumerate}
\item[\rm (i)] $M _n$ is isomorphic to a direct sum of ideals of $R$,
\item[\rm (ii)] $M _n$ has countable rank, and
\item[\rm (iii)] $M _n$ is pure in $M$.
\end{enumerate}
\end{theorem}

\begin{proof}
Fix a countable maximal independent set $\{ a _n : n \in \mathbb {Z} ^+ \}$ of $M$, and assume that we have already constructed the links of a finite ascending chain
\begin{equation}
0 = A _0 \leq A _1 \leq \dots \leq A _n,
\end{equation}
for some $n < \omega$, such that the following properties are satisfied for every $i = 1 , \dots , n$:
\begin{enumerate}
\item[(a)] $A _i$ is isomorphic to a finite direct sum of ideals of $R$,
\item[(b)] $A _i$ contains $\{ a _1 , \dots , a _i \}$, and
\item[(c)] $A _i$ is pure in $M$.
\end{enumerate}

Let $X _n$ be a maximal independent set of $A _n$, and let $k < \omega$ be such that $M _k$ contains both $X _n$ and $a _{n + 1}$. Clearly, $A _n + a _{n + 1} R$ is contained in a finite direct sum $A _{n + 1}$ of ideals in the decomposition of $M _k$. Using induction, we construct a countable ascending chain (\ref {Eq:Chain6.11}) of finite rank, pure submodules of $M$ which are finite direct sums of ideals of $R$, and whose union is equal to $M$. Lemma \ref {Lemma:6.2} implies that $A _{n + 1} = A _n \oplus B _n$, for every $n < \omega$, where $B _n$ is isomorphic to a finite direct sum of ideals of $R$. It follows that $M$ itself is isomorphic to a direct sum of ideals of $R$.
\end{proof}

A \emph {$G (\aleph _0)$-family} of an $R$-module $M$ is a family $\mathcal {B}$ consisting of submodules of $M$, with the following properties:
\begin{enumerate}
\item[(i)] $0 , M \in \mathcal {B}$,
\item[(ii)] $\mathcal {B}$ is closed under unions of ascending chains of arbitrary lengths, and
\item[(iii)] for every $A _0 \in \mathcal {B}$ and every countable set $H \subseteq M$, there exists $A \in \mathcal {B}$ containing $A _0$ and $H$, such that $A / A _0$ is countably generated.
\end{enumerate}
Clearly, an intersection of a countable number of $G (\aleph _0)$-families of submodules of $M$ is again a $G (\aleph _0)$-family of submodules of $M$. The `rank version' of this definition is called a $G (\aleph _0) ^\prime$-family. More precisely, a \emph {$G (\aleph _0) ^\prime$-family} of $M$ is a family $\mathcal {B}$ of submodules of $M$, satisfying (i) and (ii) above, in addition to the property:
\begin{enumerate}
\item[(iii)$^\prime$] for every $A _0 \in \mathcal {B}$ and every countable set $H \subseteq M$, there exists $A \in \mathcal {B}$ containing $A _0$ and $H$, such that $A / A _0$ has countable rank.
\end{enumerate}
A $G (\aleph _0)$-family of submodules of $M$ is a \emph {tight system} if, in addition, it satisfies:
\begin{enumerate}
\setlength{\itemsep}{-4pt}
\item[(iv)] for every $A \in \mathcal {B}$, $\pd _R A \leq 1$ and $\pd _R (M / A) \leq 1$.
\end{enumerate}

It is worth noticing that every module has a $G (\aleph _0)$-family of submodules, namely, the collection of all its submodules. However, not every module has a $G (\aleph _0)$-family consisting of pure submodules. Nevertheless, Bazzoni and Fuchs proved \cite {Bazzoni-Fuchs} that every torsion-free module of projective dimension at most equal to $1$ over a valuation domain has a tight system consisting of pure submodules. 

It is important to recall that valuation domains have a unique maximal ideal. Moreover, localizations of Pr\"{u}fer domains are again Pr\"{u}fer domains and, particularly, localizations of Pr\"{u}fer domains at prime ideals are valuation domains. As a consequence, every torsion-free module of projective dimension at most $1$ over a Pr\"{u}fer domain with a countable number of maximal ideals has a $G (\aleph _0)$-family consisting of pure submodules. Indeed, let $M$ be a torsion-free $R$-module of projective dimension at most $1$, where $R$ is a Pr\"{u}fer domain with a countable number of maximal ideals. For each maximal ideal $P$ of $R$, let $\mathcal {B} _P ^\prime$ be a $G (\aleph _0)$-family of pure submodules of the localization of $M$ with respect to $P$, and let $\mathcal {B} _P = \{ A \cap M : A \in \mathcal {B} _P ^\prime \}$. The desired family is the intersection of all the families $\mathcal {B} _P$.


We prove next a generalization of Hill's theorem to modules which are isomorphic to direct sums of ideals of $h$-local Pr\"{u}fer domains.

\begin{theorem}
Let $R$ be an $h$-local Pr\"{u}fer domain. A torsion-free module $M$ is isomorphic to a direct sum of ideals of $R$ if it is the union of a countable ascending chain (\ref {Eq:Chain6.5}) of submodules, such that the following properties are satisfied, for every $n < \omega$:
\begin{enumerate}
\item[\rm (i)] $M _n$ is isomorphic to a direct sum of ideals of $R$,
\item[\rm (ii)] $M _{n + 1} / M _n$ has a $G (\aleph _0)$-family $\mathcal {C} _n$ of pure submodules, and
\item[\rm (iii)] $M _n$ is pure in $M$.
\end{enumerate} \label{Thm:6.13}
\end{theorem}

For each $n < \omega$, the module $M _n$ can be written as the direct sum of $R$-modules $I _\alpha ^n$, with $\alpha$ in some set of indexes $\Omega _n$, each of which is isomorphic to an ideal of $R$. The collection $\mathcal {B} _n$ of submodules of $M$ of the form $\oplus _{\alpha \in \Lambda} I _\alpha ^n$, for some $\Lambda \subseteq \Omega _n$, is a $G (\aleph _0) ^\prime$-family of pure submodules of $M _n$, for every $n < \omega$. Moreover, the collection
\begin{equation}
\mathcal {B} _n ^\prime = \left\{ A \in \mathcal {B} _n : \frac {(A + M _j) \cap M _{j + 1}} {M _j} \in \mathcal {C} _j \textrm{, for every } j < n \right\} \quad (n < \omega)
\end{equation}
is a $G (\aleph _0) ^\prime$-family of pure submodules of $M _n$. Furthermore, the class
\begin{equation}
\mathcal {B} = \{ A \leq M : A \cap M _n \in \mathcal {B} _n ^\prime \textrm {, for every } n < \omega \}
\end{equation}
is a $G (\aleph _0) ^\prime$-family of pure submodules of $M$, such that for every $n < \omega$ and every $A \in \mathcal {B} _n$, the module $A + M _n$ is pure in $M$.

\begin{lemma}
For every $A \in \mathcal {B}$, finite rank, pure submodules of $M / A$ are isomorphic to direct sums of ideals of $R$. \label {Lemma:6.14}
\end{lemma}

\begin{proof}
Let $D$ be a submodule of $M$ containing $A$, such that $D / A$ is a finite rank, pure submodule of $M / A$, and choose a countable set of representatives $S \subseteq D$ of a maximal independent system of $D$ modulo $A$. Let $k < \omega$ be an index such that $S \subseteq M _k$. Since $A + M _k$ and $D$ are pure in $M$, then $A + (D \cap M _k) = D \cap (A + M _k)$ is a pure submodule between $A$ and $D$ which contains $S$, so that $D = A + (D \cap M _k)$. Now, $A \cap M _k$ belongs to $\mathcal {B} _k ^\prime$ and, consequently, it is a direct summand of $M _k$, say, $M _k = (A \cap M _k) \oplus B$, for some submodule $B$ of $M _k$. Therefore, $D \cap M _k = (A \cap M _k) \oplus (B \cap D)$.

Being isomorphic to a submodule of a finite rank module, $C = B \cap D$ is a finite rank, pure submodule in the module $M _k$ which, in turn, is isomorphic to a direct sum of ideals. So, $C$ itself is isomorphic to a direct sum of ideals by Lemma \ref {Lemma:6.2}. Moreover, since
\begin{equation}
D = A + (D \cap M _k) = (A + (A \cap M _k)) \oplus C = A \oplus C,
\end{equation}
$D / A$ is isomorphic to a direct sum of ideals.
\end{proof}

The following result is a consequence of the proof of Lemma \ref {Lemma:6.14}.

\begin{lemma}
Every $A \in \mathcal {B}$ is a balanced submodule of $M$.
\end{lemma}

\begin{proof}
Since $A$ is a direct summand of every submodule $D$ of $M$ for which $D / A$ has rank $1$, then $A$ is balanced in $D$. The conclusion of this result follows now from Lemma \ref {Lemma:5.1}.
\end{proof}

We are now in a position to prove our generalization of Hill's theorem to modules which are isomorphic to direct sums of ideals of $h$-local Pr\"{u}fer domains.

\begin{proof}[Proof of Theorem \ref {Thm:6.13}]
Let $\alpha$ be any nonzero ordinal, and assume that we have already constructed the links of the continuous, well-ordered, ascending chain
\begin{equation}
0 = A _0 \leq A _1 \leq \dots \leq A _\gamma \leq A _{\gamma + 1} \leq \dots \quad (\gamma < \alpha) \label {Eq:Chain6.13}
\end{equation}
of modules of $\mathcal {B}$, such that $A _{\gamma + 1} / A _\gamma$ is isomorphic to a direct sum of ideals of $R$, for every $\gamma < \alpha$. If $\alpha$ is a limit ordinal, we let $A _\alpha$ be the union of the links of chain (\ref {Eq:Chain6.13}). Otherwise, $\alpha$ is the successor of an ordinal number $\beta$. If there exists $x \in M \setminus A _\beta$, then there exists $A _\alpha \in \mathcal {B}$ containing both $x$ and $A _\beta$, such that $A _\alpha / A _\beta$ has countable rank. Finite rank, pure submodules of $A _\alpha / A _\beta$ are isomorphic to direct sums of ideals of $R$ by Lemma \ref {Lemma:6.14}, thus $A _\alpha / A _\beta$ itself is a direct sum of ideals of $R$ by Lemma \ref {Lemma:6.5}. In such way, we construct a continuous, well-ordered, ascending chain (\ref {Eq:Chain6.9}) satisfying the hypothesis of Lemma \ref {Lemma:6.9}. We conclude that $M$ is isomorphic to a direct sum of ideals of $R$.
\end{proof}

Since Dedekind domains are $h$-local Pr\"{u}fer domains for which every ideal is projective, we have the following obvious improvement of Theorem \ref {Thm:6.13}.

\begin{corollary}
Let $R$ be a Dedekind domain with a countable number of maximal ideals. The torsion-free module $M$ is isomorphic to a direct sum of ideals of $R$ if it is the union of a countable ascending chain (\ref {Eq:Chain6.5}) of submodules, such that the following are satisfied, for every $n < \omega$:
\begin{enumerate}
\item[\rm (i)] $M _n$ is isomorphic to a direct sum of ideals of $R$, and
\item[\rm (ii)] $M _n$ is pure in $M$.
\end{enumerate}
\label{Coro:Final}
\end{corollary}

\begin{proof}
For every $n < \omega$, $M _{n + 1} / M _n$ is a torsion-free module of projective dimension less than or equal to $1$. The observation preceding Theorem \ref {Thm:6.13} implies now that every factor module of (\ref {Eq:Chain6.5}) admits a $G (\aleph _0)$-family of pure submodules. By Theorem \ref {Thm:6.13}, $M$ is isomorphic to a direct sum of ideals of $R$.
\end{proof}

It is worth noticing that the conclusion of Corollary \ref{Coro:Final} is reached if, in particular, the links $M _n$ are free modules. Moreover, it is easy to check that if $M$ is a torsion-free module over an integral domain $R$ which is the union of the countable ascending chain (\ref {Eq:Chain6.5}) of projective, pure submodules, then there exists a chain
\begin{equation}
0 = F _0 \leq F _1 \leq \dots \leq F _n \leq \dots \quad (n < \omega) \label {Chain:Fs}
\end{equation}
consisting of free $R$-modules, such that:
\begin{enumerate}
\item[\rm (a)] every $F _n$ is pure in $F _{n + 1}$,
\item[\rm (b)] every $F _n$ contains $M _n$ as a direct summand, say, $F _n = M _n \oplus K _n$,
\item[\rm (c)] $M$ is a direct summand of $F = \bigcup _{n < \omega} F _n$, and 
\item[\rm (d)] $\{ K _n \} _{n < \omega}$ may be chosen to form an ascending chain under inclusion.
\end{enumerate} 
The proof of the next result is now straight-forward.

\begin{corollary}
Let $R$ be a Dedekind domain with a countable number of maximal ideals. The torsion-free module $M$ is projective if it is the union of a countable ascending chain (\ref {Eq:Chain6.5}) of projective, pure submodules. \qed
\end{corollary}

\subsubsection*{Acknowledgments}

The author wishes to acknowledge the guidance of Prof. L\'{a}szl\'{o} Fuchs at every stage of this investigation. Also, he wishes to thank the anonymous reviewer for her/his kind comments, which led to improve the quality of this work. The results presented here are part of research project PIM10-01 at the Universidad Aut\'{o}noma de Aguascalientes.

\end{document}